\pdfoutput=1
\RequirePackage{ifpdf}
\ifpdf 
\documentclass[pdftex]{sigma}
\else
\documentclass{sigma}
\fi

\usepackage[all]{xy}
\xyoption{arrow}
\xyoption{graph}

\numberwithin{equation}{section}

\newtheorem{Theorem}{Theorem}[section]
\newtheorem{Lemma}[Theorem]{Lemma}
\newtheorem{Proposition}[Theorem]{Proposition}
 { \theoremstyle{definition}
\newtheorem{Definition}[Theorem]{Definition}
\newtheorem{Remark}[Theorem]{Remark} }

\begin{document}
\allowdisplaybreaks

\newcommand{\arXivNumber}{2003.05890}

\renewcommand{\thefootnote}{}

\renewcommand{\PaperNumber}{069}

\FirstPageHeading

\ShortArticleName{On the Irreducibility of Some Quiver Varieties}

\ArticleName{On the Irreducibility of Some Quiver Varieties\footnote{This paper is a~contribution to the Special Issue on Noncommutative Manifolds and their Symmetries in honour of Giovanni Landi. The full collection is available at \href{https://www.emis.de/journals/SIGMA/Landi.html}{https://www.emis.de/journals/SIGMA/Landi.html}}}

\Author{Claudio BARTOCCI~$^{\dag^1\dag^2}\!$, Ugo BRUZZO~$^{\dag^3\dag^4 \dag^5 \dag^6\dag^7}\!$, Valeriano LANZA~$^{\dag^8}\!$ and Claudio L.S.~RAVA~$^{\dag^1}\!\!\!$}

\AuthorNameForHeading{C.~Bartocci, U.~Bruzzo, V.~Lanza and C.L.S.~Rava}

\Address{$^{\dag^1}$~Dipartimento di Matematica, Universit\`a di Genova, Via Dodecaneso 35, 16146 Genova, Italy}
\EmailDD{\href{mailto:bartocci@dima.unige.it}{bartocci@dima.unige.it}, \href{mailto:clsrava@gmail.com}{clsrava@gmail.com}}
\URLaddressDD{\url{http://www.dima.unige.it/~bartocci/index.htm}}

\Address{$^{\dag^2}$~Laboratoire SPHERE, CNRS, Universit\'e Paris Diderot (Paris 7), 75013 Paris, France}

\Address{$^{\dag^3}$~SISSA (Scuola Internazionale Superiore di Studi Avanzati),\\
\hphantom{$^{\dag^3}$}~Via Bonomea 265, 34136 Trieste, Italy}
\EmailDD{\href{mailto:bruzzo@sissa.it}{bruzzo@sissa.it}}
\URLaddressDD{\url{http://www.people.sissa.it/~bruzzo/webpage/}}

\Address{$^{\dag^4}$~Departamento de Matem\'atica, Universidade Federal da Para\'iba,\\
\hphantom{$^{\dag^4}$}~Campus I, Jo\~ao Pessoa, PB, Brasil}

\Address{$^{\dag^5}$~IGAP (Institute for Geometry and Physics), Trieste, Italy}

\Address{$^{\dag^6}$~INFN (Istituto Nazionale di Fisica Nucleare), Sezione di Trieste, Italy}

\Address{$^{\dag^7}$~Arnold-Regge Center for Algebra, Geometry and Theoretical Physics, Torino, Italy}

\Address{$^{\dag^8}$~Departamento de An\'alise, IME, Universidade Federal Fluminense,\\
\hphantom{$^{\dag^8}$}~Rua Professor Marcos Waldemar de Freitas Reis, Niter\'oi, RJ, Brazil}
\EmailDD{\href{mailto:vlanza@id.uff.br}{vlanza@id.uff.br}}

\ArticleDates{Received March 13, 2020, in final form July 10, 2020; Published online July 26, 2020}

\Abstract{We prove that certain quiver varieties are irreducible and therefore are isomorphic to Hilbert schemes of points of the total spaces of the bundles $\mathcal O_{\mathbb P^1}(-n)$ for $n \ge 1$.}


\Classification{14D20; 14D21; 14J60; 16G20}

\begin{flushright}
\begin{minipage}{96mm}
\it We thank Gianni Landi for the long friendship and for all \\ collaborations that took place over the last 35 years
\end{minipage}
\end{flushright}

\renewcommand{\thefootnote}{\arabic{footnote}}
\setcounter{footnote}{0}

\section{Introduction}

 Nakajima's quiver varieties were introduced by Hiraku Nakajima in \cite{Naka-ALE} to study the moduli spaces of instantons on ALE spaces, and have been extensively studied since then, see, e.g., \cite{Ginz,Kuz,Nakabook,Naka-ring}. They provide a modern and significant example of how algebra and geometry can be sometimes so deeply, yet surprisingly connected: in fact, their main feature is that they allow one to put in relation some moduli spaces of bundles (or torsion-free sheaves) over certain smooth projective varieties with some moduli spaces of representations of suitable algebras (the so-called \emph{path algebras} of a~quiver and quotients of them). A major example of this bridge is given by the moduli space of framed sheaves on~$\mathbb{P}^2$, which can be identified with the moduli space of semistable representations of the ADHM quiver (see~\cite{Nakabook} for details).

 The way this relation is usually looked at is the one that inspired Nakajima's first pioneering work: the philosophy is to use the algebraic data we get on one side (usually called \emph{ADHM data}) to parameterize the geometric moduli spaces we have on the other side, i.e., the objects we are actually interested in (see for example \cite{BFMT, Dorey,Niki}). But sometimes it may be useful to switch roles and use the geometric interpretation as a ``tool'' to prove something interesting per se on the algebraic side. For instance, this is the case when one deals with irreducibility problems: to determine whether a variety of matrices is irreducible is known to be a challenging problem (see~\cite{Ngo-Siv} and references therein), and in the specific case of Nakajima's quiver varieties the conclusive result by Crawley-Boevey stating that \emph{all of them} are indeed irreducible has been achieved only by using hyperk\"ahler geometry techniques~\cite{CrBo}.

 In \cite{bblr} we introduced a collection of new quiver varieties $\mathcal{M} (\Lambda_n,\vec{v}_c,w_c, \vartheta_c)$, ${n\geq1}$ (see below for the notation);
 for $n\neq 2$ they are not Nakajima's quiver varieties, as the quivers involved are not doubles. We proved that~$\mathcal{M} (\Lambda_1,\vec{v}_c,w_c, \vartheta_c)$ is isomorphic to the Hilbert scheme of points of the total space of $\mathcal{O}_{\mathbb{P}^1}(-1)$, and, in particular, that it is therefore irreducible (as the Hilbert scheme is so~\cite{Fogarty}). For $n\geq 2$ we only proved a weaker result, i.e.,
 that only a certain connected component of $\mathcal{M} (\Lambda_n,\vec{v}_c,w_c, \vartheta_c)$ can be identified with $\operatorname{Hilb}^c(\operatorname{Tot}(\mathcal{O}_{\mathbb{P}^1}(-n)))$. However, as $\mathcal{M} (\Lambda_2,\vec{v}_c,w_c, \vartheta_c)$ is a Nakajima quiver variety, its irreducibility follows from Crawley-Boevey's result, so that one only has to determine whether the varieties $\mathcal{M} (\Lambda_n,\vec{v}_c,w_c, \vartheta_c)$ are irreducible for $n\geq 3$. In this paper we prove this fact, completing the work of \cite{bblr}, actually showing directly
 that $\operatorname{Hilb}^c(\operatorname{Tot}(\mathcal{O}_{\mathbb{P}^1}(-n)))$ is isomorphic to the whole
$\mathcal{M} (\Lambda_n,\vec{v}_c,w_c, \vartheta_c)$. As this technique also works for the case $n=2$ we include it as well.

\section{Some background } \label{background}
The quivers we are going to consider are extracted from the ADHM data for the Hilbert schemes of points of the varieties $\operatorname{Hilb}^c(X_n)$,
 where $X_n$ is the total space of the line bundle $\mathcal{O}_{\mathbb{P}^1}(-n)$,
and, in turn,
the construction of the ADHM data is based on the description of the moduli spaces of framed sheaves on the Hirzebruch surfaces $\Sigma_n$ in terms of monads that was given in~\cite{bbr}.
We denote by $H$ and $E$ the classes in $\operatorname{Pic}(\Sigma_n)$ of the sections
of the natural ruling $\Sigma_n\to\mathbb{P}^1$ that square to
$n$ and $-n$, respectively.
We fix a curve $\ell_{\infty}$ in $\Sigma_n$ belonging to the class $H$ (the ``line at infinity'').
A framed sheaf on $\Sigma_n$ is a pair $(\mathcal{E}, \theta)$, where $\mathcal{E}$ is a rank $r$ torsion-free sheaf which
is trivial along $\ell_{\infty}$, and $\theta \colon \mathcal{E}\vert_{\ell_{\infty}}\stackrel{\sim}{\to}\mathcal{O}_{\ell_{\infty}}^{\oplus r}$ is an isomorphism. A morphism between framed sheaves $(\mathcal{E}, \theta)$, $(\mathcal{E}', \theta')$ is by definition a morphism $\Lambda\colon \mathcal{E} \longrightarrow \mathcal{E}'$ such that
$\theta'\circ\Lambda\vert_{\ell_{\infty}} = \theta$. The moduli space parameterizing isomorphism classes of framed sheaves $(\mathcal{E}, \theta)$ on $\Sigma_n$ with Chern character $\textrm{ch}(\mathcal{E}) = \big(r, aE, -c -\frac{1}{2} na^2\big)$, where $r, a, c \in \mathbb{Z}$ and $r\geq 1$, will be denoted $\mathcal{M}^{n}(r,a,c)$. We normalize the framed sheaves so that $0\leq a\leq r-1$.

A monad $M$ on a scheme $X$ is a three-term complex of locally free $\mathcal O_X$-modules of finite rank, having nontrivial cohomology only at the middle term (cf.~\cite[Definition~II.3.1.1]{Ok}).
It was proved in~\cite{bbr} that a~framed sheaf
$(\mathcal{E}, \theta)$ on $\Sigma_n$ with invariants $(r,a,c)$ is the cohomology of a~monad
\begin{equation}
M(\alpha,\beta)\colon \ \xymatrix{0 \ar[r] & \mathcal{U}_{\vec{k}} \ar[r]^-{\alpha} & \mathcal{V}_{\vec{k}} \ar[r]^-{\beta} & \mathcal{W}_{\vec{k}} \ar[r] & 0
} , \label{fundamentalmonad}
\end{equation}
where $\vec{k}$ is the quadruple $ (n,r,a,c)$, and
\begin{equation*}
\mathcal{U}_{\vec{k}}:=\mathcal{O}_{\Sigma_n}(0,-1)^{\oplus k_1},\qquad
\mathcal{V}_{\vec{k}}:=\mathcal{O}_{\Sigma_n}(1,-1)^{\oplus k_2} \oplus \mathcal{O}_{\Sigma_n}^{\oplus k_4},\qquad
\mathcal{W}_{\vec{k}}:=\mathcal{O}_{\Sigma_n}(1,0)^{\oplus k_3} ,
\end{equation*}
with
\begin{equation*}
 k_1=c+\dfrac{1}{2}na(a-1),\qquad
k_2=k_1+na,\qquad
k_3=k_1+(n-1)a,\qquad
k_4=k_1+r-a .
\end{equation*}

The space $L_{\vec{k}}$ of pairs in $\operatorname{Hom}(\mathcal{U}_{\vec{k}},\mathcal{V}_{\vec{k}})\oplus\operatorname{Hom}(\mathcal{V}_{\vec{k}},\mathcal{W}_{\vec{k}})$ fitting into~\eqref{fundamentalmonad},
such that the cohomology of the complex is torsion-free and trivial at infinity, is a smooth algebraic variety.
There is a principal $\operatorname{GL}(r,\mathbb{C})$-bundle $P_{\vec{k}}$ over $L_{\vec{k}}$ whose fibre at a point $(\alpha,\beta)$ is the space of framings for the corresponding cohomology of~\eqref{fundamentalmonad}.
The algebraic group
\[
G_{\vec{k}}=\operatorname{Aut}(\mathcal{U}_{\vec{k}})\times\operatorname{Aut}(\mathcal{V}_{\vec{k}})\times\operatorname{Aut}(\mathcal{W}_{\vec{k}})
\]
acts freely on $P_{\vec{k}}$, and the moduli space $\mathcal{M}^{n}(r,a,c)$ is the quotient $P_{\vec{k}}/G_{\vec{k}}$ \cite[Theorem 3.4]{bbr}. This is nonempty if and only if $c + \frac{1}{2} na(a-1) \geq 0$, and when nonempty, it is a smooth algebraic variety of dimension $2rc + (r-1) na^2$.

When $r=1$ we can assume $a=0$, and there is an identification
\begin{equation*}
\mathcal{M}^{n}(1,0,c) \simeq \operatorname{Hilb}^c (\Sigma_n \setminus \ell_{\infty}) = \operatorname{Hilb}^c (X_{n}) .
\end{equation*}

A first step to construct ADHM data for the Hilbert schemes of points of the varieties $X_n$
is to show that the Hilbert schemes can be covered by open subsets that are isomorphic to the Hilbert scheme of $\mathbb C^2$,
and therefore have an ADHM description, according to Nakajima. Then one proves that these ``local ADHM data'' can be glued to provide ADHM data for the Hilbert schemes of $X_n$.

Let $P^{n}(c)$ be the set of collections $(A_1,A_2;C_1,\dots,C_{n};e)$
in $\operatorname{End}(\mathbb{C}^{c})^{\oplus n+2}\oplus\operatorname{Hom}(\mathbb{C}^{c},\mathbb{C})$ sa\-tisfying the conditions\\[8pt]
(P1)\ \ \
$\displaystyle
\begin{cases}
A_1C_1A_{2}=A_2C_{1}A_{1},&\qquad\text{when $n=1$},\\[2pt]
\begin{aligned}
A_1C_q&=A_2C_{q+1},\\
C_qA_1&=C_{q+1}A_2
\end{aligned}
\qquad\text{for}\quad q=1,\dots,n-1,&\qquad\text{when $n>1$;}
\end{cases}
$
\\[8pt]
(P2) \
$A_1+\lambda A_2$ is a \emph{regular pencil} of matrices, i.e., there exists $[\nu_{1},\nu_{2}]\in\mathbb{P}^1$ such that $\det(\nu_1A_1+\nu_2A_2)\neq0$;
\\[8pt]
(P3) \
for all values of the parameters $ ([\lambda_1,\lambda_2],(\mu_1,\mu_{2}) )\in\mathbb{P}^1\times\mathbb{C}^{2}$ satisfying
\begin{equation*}
\lambda_{1}^{n}\mu_{1}+\lambda_{2}^{n}\mu_{2}=0
\end{equation*}
there is no nonzero vector $v\in\mathbb{C}^c$ such that
\begin{equation*}
\begin{cases}
C_{1}A_{2}v=-\mu_1v,\\
C_{n}A_{1}v=(-1)^n\mu_2v,\\
v\in\ker e
\end{cases}
\qquad\text{and}\qquad\left(\lambda_2{A_1}+\lambda_1{A_2}\right)v=0 .
\end{equation*}

The group
$\operatorname{GL}(c, \mathbb{C})\times \operatorname{GL}(c, \mathbb{C})$ acts on $P^n(c)$ according to
\begin{equation*} (A_i, C_j, e ) \mapsto \big(\phi_2A_i\phi_1^{-1},\phi_1C_j\phi_2^{-1}, e\phi_1^{-1}\big)
\end{equation*}
for $i=1,2$, $j=1,\dots,n$, $ (\phi_1,\phi_2)\in\operatorname{GL}(c, \mathbb{C})\times \operatorname{GL}(c, \mathbb{C})$.

 The following result expresses the fact that the collections $(A_1,A_2;C_1,\dots,C_{n};e)$ satisfying
 conditions (P1) to (P3) are ADHM data for the varieties $\operatorname{Hilb}^c (X_{n})$ (this is Theorem~3.1 in~\cite{bbr}).
 \begin{Theorem}\label{thm0}
$P^{n}(c)$ is a principal $\operatorname{GL}(c, \mathbb{C})\times \operatorname{GL}(c, \mathbb{C})$-bundle over $\operatorname{Hilb}^c (X_{n})$.
\end{Theorem}

\section{The main result}\label{sectionmainresult}

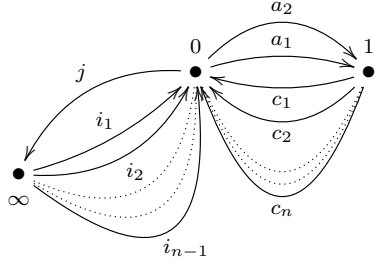
\begin{figure}[t]\centering
$ \xymatrix@R-2.3em{
&&\mbox{\scriptsize$0$}&&\mbox{\scriptsize$1$}\\
&&\bullet\ar@/_3ex/[llddddddd]_{j}\ar@/^/[rr]^{a_{1}}\ar@/^4ex/[rr]^{a_{2}}&&\bullet\ar@/^/[ll]^{c_{1}}
\ar@/^4ex/[ll]^{c_2}\ar@{..}@/^8ex/[ll]\ar@{..}@/^9ex/[ll]
\ar@/^10ex/[ll]^{c_{n}}& \\ \\ \\ \\ \\&&&&& \\ \\
\bullet\ar@/_/[rruuuuuuu]^{i_1}\ar@/_3ex/[rruuuuuuu]_{i_2}\ar@{..}@/_6ex/[rruuuuuuu]\ar@{..}@/_8ex/[rruuuuuuu]\ar@/_10ex/[rruuuuuuu]_{i_{n-1}}&&&&&\\
\mbox{\scriptsize$\infty$}&&&&
}
$
\caption{The quivers $Q_n$.}\label{Qn}
\end{figure}
Now we turn to the purpose of this paper, namely, proving that the Hilbert schemes of points of the varieties $X_n$ are
isomorphic to moduli spaces of representations of suitable quivers.
 For any $n\geq2$ let $Q_{n}$ be the framed quiver in Fig.~\ref{Qn}, where $\infty$ is the framing vertex. Let $J_{n}$ be the two sided ideal of $\mathbb{C} Q_{n}$ generated by the relations
\begin{equation}
\begin{cases}
a_2c_{q+1}-a_1c_q =0,\\
c_{q+1}a_2-c_qa_1-i_qj =0
\end{cases}
\qquad\text{for}\quad q=1,\dots,n-1 .
\label{eq-gen-J}
\end{equation}
Our purpose is to describe the spaces of representations of the quiver $Q_n$ with relations $J_{n}$, i.e., the spaces of representations of the quotient algebra $\Lambda_{n}=\mathbb{C} Q_{n}/J_{n}$.

We recall some basic definitions. Given $\vec{v} = (v_0, v_1)\in \mathbb{N}^2$ and $w\in \mathbb{N}$,
a {\it $(\vec{v},w)$-dimensional representation of $\Lambda_n$}
is the datum of a triple of $\mathbb{C}$-vector spaces $V_0$, $V_1$, $W$, with $\dim V_i = v_i$, $\dim W=w$, and of an element
$(A_1,A_2;C_1, \dots,C_n;e;f_1,\dots,f_{n-1})$ in
\[
\operatorname{Hom}_\mathbb{C}(V_0,V_1)^{\oplus2}\oplus\operatorname{Hom}_\mathbb{C}(V_1,V_0)^{\oplus n}\oplus\operatorname{Hom}_\mathbb{C}(V_0,W)\oplus\operatorname{Hom}_\mathbb{C}(W,V_0)^{\oplus n-1}
\]
 satisfying the relations determined by equations~\eqref{eq-gen-J}, namely
 \begin{equation}
\begin{cases}
A_2 C_{q+1}- A_1 C_q =0,\\
C_{q+1} A_2 - C_q A_1- f_q e =0
\end{cases}
\qquad\text{for}\quad q=1,\dots,n-1 .
\label{eq-gen-Q1}\tag{Q1}
\end{equation}

The space $\operatorname{Rep}(\Lambda_n ,\vec{v},w)$ of all $(\vec{v},w)$-dimensional representations of $\Lambda_n$
is an affine variety, on which the group $G_{\vec{v}}=\operatorname{GL}(v_0,\mathbb{C})\times\operatorname{GL}(v_1,\mathbb{C})$ acts by basis change. Indeed, we ignore the action of $\operatorname{GL}(w,\mathbb{C})$ on the vector space~$W$ attached to the framing vertex. As usual, to get a~well behaved quotient space one has to perform a GIT construction by introducing a suitable notion of stability. This was done by A.~King \cite{KiQ} and, in a slightly different way, by A.~Rudakov~\cite{Rud97}. In the case of a quiver with a framing vertex, the following definition can be shown to be equivalent to the King--Rudakov one \cite{blr, CrBo}.
\begin{Definition}\label{def:stabil}
 Fix $\vartheta\in\mathbb R^2$. A $(\vec{v},w)$-dimensional representation {$(V_0,V_1,W)$} of $\Lambda_n$ is said to be \emph{$\vartheta$-semistable} if, for any subrepresentation $S=(S_0,S_1)$ {$\subseteq (V_0,V_1)$}, one has:
\begin{gather}
 \text{if $S_0\subseteq \ker e$, then $\vartheta\cdot(\dim S_0,\dim S_1)\leq0$} ;\label{eq:stabnak1}\\[5pt]
 \text{if $S_0\supseteq \operatorname{Im} f_i$ for $i=1,\dots,n-1$, then $\vartheta\cdot(\dim S_0,\dim S_1)\leq\vartheta\cdot(v_0,v_1)$} .\label{eq:stabnak2}
\end{gather}
A $\vartheta$-semistable representation is \emph{$\vartheta$-stable} if a strict inequality holds in \eqref{eq:stabnak1} whenever $S\neq0$ and
in \eqref{eq:stabnak2} whenever $S\neq(V_0,V_1)$.
\end{Definition}
Let $\operatorname{Rep}(\Lambda_n, \vec{v}, w)^{\rm ss}_\vartheta$ be the subset of $\operatorname{Rep}(\Lambda_n,\vec{v},w)$ consisting of $\vartheta$-semistable representations. By \cite[Proposition 5.2]{KiQ}, the coarse moduli space of
$(\vec{v},w)$-dimensional $\vartheta$-semistable representations of $\Lambda_n$ is the GIT quotient
\[\mathcal{M} (\Lambda_n,\vec{v},w, \vartheta)=\operatorname{Rep}(\Lambda_n, \vec{v}, w)^{\rm ss}_\vartheta/\!/ G_{\vec{v}} .\]
It can be proved that the open subset $\mathcal{M}^{\rm s} (\Lambda_n,\vec{v},w, \vartheta) \subset \mathcal{M} (\Lambda_n,\vec{v},w, \vartheta)$
consisting of stable representations makes up a fine moduli space. Notice that, for quivers without a framing, this
holds only when the dimension vector is primitive \cite[Proposition~5.3]{KiQ}, whilst this requirement is not necessary in the case of framed quivers~\cite{CrBo}. Theorem~4.5 of~\cite{bblr} states that the Hilbert scheme of points
$\operatorname{Hilb}^c (X_n)$ can be embedded into $\mathcal{M} (\Lambda_n,\vec{v},w, \vartheta)$
for suitable choices of $\vec{v}$, $w$, and $\vartheta$. Precisely, one has the following result:
\begin{Theorem}\label{thm45}
For every $n\geq 2$ and $c\geq 1$ let
\[\vec{v}_c=(c,c) ,\qquad w_c = 1 ,\qquad \vartheta_c = (2c, 1-2c ),\]
and let
$\mathcal{H} (n,c) $ be the irreducible component of $\mathcal{M}(\Lambda_n,\vec{v}_c, 1, \vartheta_c) $
given by the equations
\begin{equation}\label{sliceeqs}
f_1 = f_2 = \cdots = f_{n-1} =0 .
\end{equation}
Then $\operatorname{Hilb}^c (X_n)\simeq \mathcal{H} (n,c) $.
\end{Theorem}
Let $\mbox{pr}\colon \operatorname{Rep}(\Lambda_n, \vec{v}_c, 1)^{\rm ss}_{\vartheta_c} \to \mathcal{M} (\Lambda_n,\vec{v}_c,1, \vartheta_c) $
be the quotient map. The proof {of Theorem~\ref{thm45} basically consists in proving that the counterimage
$\operatorname{pr}^{-1}(\mathcal{H} (n,c) )=:Z_n(c)$ coincides with the total space of the principal fibration $P^n(c)$ we introduced in Section~\ref{background}. As it is quite involved and requires a few intermediate Lemmas and Propositions, we refer the reader to~\cite{bblr} for further details.} Here we only note that the starting point is given by the stability conditions in Definition~\ref{def:stabil}.

\begin{Remark}The set of $(\vec{v}_c,w_c)$-dimensional representations of $\Lambda_n$ which are semistable according to Definition~\ref{def:stabil} does not change if we let the stability parameter vary inside the open cone
\[
 \Gamma_c=\left\{\vartheta=(\vartheta_0,\vartheta_1)\in\mathbb{R}^2\,|\,\vartheta_0>0,\, -\vartheta_0<\vartheta_1<-\frac{c-1}{c}\vartheta_0\right\} .
\]
It can be shown that
for any stability parameter $\bar{\vartheta}$ on the open rays
\begin{gather*}
 R_1 = \big\{(\vartheta_0,\vartheta_1)\in\mathbb{R}^2\,|\,\vartheta_0>0 ,\, \vartheta_0+\vartheta_1=0\big\} ,\\
 R_2 = \big\{(\vartheta_0,\vartheta_1)\in\mathbb{R}^2\,|\,\vartheta_0>0 ,\, (c-1)\vartheta_0+c\vartheta_1=0\big\}
\end{gather*}
there exist representations which are $\bar{\vartheta}$-semistable, but not $\vartheta_c$-semistable.
So, $ \Gamma_c$ is a chamber in the space
$\mathbb{R}^2_{(\vartheta_0,\vartheta_1)}$ of stability parameters and the closed rays $\overline{R_1}$, $\overline{R_2}$ are its walls. Furthermore, inside $\Gamma_c$ semistability and stability are equivalent (cf.~\cite[Lemma~4.7]{bblr}): in particular, points in $\mathcal{M} (\Lambda_n,\vec{v}_c,1, \vartheta_c)$ can be thought of as
$G_{\vec{v}_c}$-orbits of representations in $\operatorname{Rep}(\Lambda_n, \vec{v}_c, 1)$.

A full description of the chamber/wall decomposition of the space $\mathbb{R}^2_{(\vartheta_0,\vartheta_1)}$ will be the object of a future work.
\end{Remark}

We wish to prove that the component $\mathcal{H} (n,c) $ of $\mathcal{M}(\Lambda_n,\vec{v}_c, 1, \vartheta_c) $ introduced in Theorem~\ref{thm45} coincides with the whole moduli space $\mathcal{M}(\Lambda_n,\vec{v}_c, 1, \vartheta_c)$ (this will be Theorem~\ref{mainthm}). Let us introduce the following notation
\begin{equation*}
\mathcal{R}(\Lambda_n,c)= \operatorname{Rep}(\Lambda_n,\vec{v}_c,1) ;\qquad \mathcal{R}^{\rm ss}(\Lambda_n,c) =\operatorname{Rep}(\Lambda_n,\vec{v}_c,1)_ {\vartheta_c}^{\rm ss} .
\end{equation*}
Given a representation $(A_1,A_2;C_1, \dots,C_n;e;f_1,\dots,f_{n-1}) \in \mathcal{R}(\Lambda_n,c)$, we form the pencil $A_1 + \lambda A_2$, with $\lambda\in \mathbb{C}$. We recall that a pencil $A_1 + \lambda A_2$ is {\it regular}
if there is a point $[\nu_1,\nu_2]\in \mathbb{P}^1$ such that $\det (\nu_1 A_1 + \nu_2 A_2) \neq 0$.

\begin{figure}[t]\centering
$ \xymatrix@R-2.3em{
&&\mbox{\scriptsize$0$}&&&&\mbox{\scriptsize$1$}\\
&&\bullet\ar@/_3ex/[llddddddd]_{j}\ar@/^/[rrrr]^{a_{1}}\ar@/^4ex/[rrrr]^{a_{2}}&&&&\bullet\ar@/^/[llll]^{b_{1}}
\ar@/^4ex/[llll]^{b_2}\ar@{..}@/^8ex/[llll]\ar@{..}@/^9ex/[llll]
\ar@/^10ex/[llll]^{b_{n-1}} \ar@/^16ex/[llll]^{d_{2}} \ar@/^19ex/[llll]^{d_{3}} \ar@{..}@/^23ex/[llll] \ar@{..}@/^25ex/[llll] \ar@/^28ex/[llll]^{d_{n}} \\ \\ \\ \\ \\ \\ \\
\bullet\ar@/_/[rruuuuuuu]^{i_1}\ar@/_3ex/[rruuuuuuu]_{i_2}\ar@{..}@/_6ex/[rruuuuuuu]\ar@{..}@/_8ex/[rruuuuuuu]\ar@/_10ex/[rruuuuuuu]_{i_{n-1}}&&&&&\\
\mbox{\scriptsize$\infty$}&&&&
}
$
\caption{The quivers $Q'_n$ for $n\ge 3$.}\label{Q'n}
\end{figure}
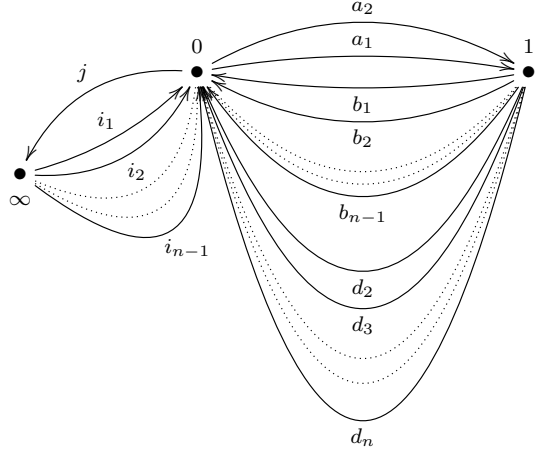

To prove Theorem \ref{mainthm} it is convenient to introduce ``augmented'' framed quivers defined as follows:
let $Q'_{2}= Q_2$, and, for every $n\geq 3$, let $Q'_{n}$ be the framed quiver in Fig.~\ref{Q'n}. Let $J'_{2}= J_2$ and, for all $n\geq 3$, let $J'_{n}$ be the two sided ideal of $\mathbb{C} Q'_{n}$ generated by the relations
\begin{equation}
\begin{cases}
a_2d_{q+1}-a_1b_q =0, \\
d_{q+1}a_2-b_qa_1-i_qj =0
\end{cases}\qquad\text{for}\quad q=1,\dots,n-1 .
\label{eq-gen-J'}
\end{equation}
We set $\Lambda'_{n}=\mathbb{C} Q'_{n}/J'_{n}$ for all $n\geq2$. Notice that $\Lambda'_{2}=\Lambda_{2}$;
for $n\geq 3$, the algebra $\Lambda_{n}$ can be obtained by taking the quotient of $\Lambda'_{n}$ by a suitable ideal.
Indeed, let $K_{n}$ be the two sided ideal of $\Lambda'_{n}$ generated by the relations
\begin{equation}
\bar{b}_{q}=\bar{d}_q\qquad\text{for}\quad q=2,\dots, n-1 ,\label{eq-gen-K_n}
\end{equation}
where $\bar{x}$ is the class in $\Lambda'_{n}$ of the element $x \in \mathbb{C} Q'_{n}$.
Let $\tilde{p}_{n}\colon\mathbb{C} Q'_{n}\longrightarrow\mathbb{C} Q_{n}$ be the $\mathbb{C}$-algebra morphism determined by the assignments
\begin{equation}\label{eq_tp_n}
\tilde{p}_{n}(a_q) = a_q , \qquad \tilde{p}_{n}(b_q) = c_q ,\qquad \tilde{p}_{n}(d_q) = c_{q} ,\qquad
\tilde{p}_{n}(j) = j , \qquad \tilde{p}_{n}(i_q) = i_q .
\end{equation}
It is straightforward that $\tilde{p}_{n}$ is surjective and that its kernel is the two sided ideal $L_{n}\subset \mathbb{C} Q'_{n}$ generated by the relations
\begin{equation}
b_{q}=d_{q}\qquad\text{for}\quad q=2,\dots,n-1 .
\label{eq-gen-L_n}
\end{equation}
It follows directly from equation~\eqref{eq_tp_n} that $\tilde{p}_n$ maps the set of generators of $J'_n$ (see equation~\eqref{eq-gen-J'}) onto the set of generators of $J_n$ (see equation~\eqref{eq-gen-J}), so that
\begin{equation*}
\tilde{p}_{n}(J'_{n}) =J_{n} .
\end{equation*}
Then it is not hard to check that $\tilde{p}_{n}$ induces a surjective morphism $p_n\colon \Lambda'_n \to \Lambda_n$,
whose kernel, by equations~\eqref{eq-gen-K_n} and \eqref{eq-gen-L_n}, is
\begin{equation*}
\ker p_{n}= L_{n}/(L_{n}\cap J'_{n})=K_{n} .
\end{equation*}
In conclusion, we have proved the following lemma.
\begin{Lemma}\label{lm-Lambda'-->>Lambda}
There is an isomorphism of $\mathbb{C}$-algebras $\Lambda'_{n}/K_{n}\simeq \Lambda_{n}$.
\end{Lemma}

One of the reasons to introduce the augmented quivers $Q'_{n}$ is that their path algebras carry
an action of the group ${\rm SO}(2, \mathbb{C})$ which descends to the quotient algebra $\Lambda'_{n}$. This
action will be instrumental in proving the regularity of the pencil $A_1 + \lambda A_2$.

Elements of ${\rm SO}(2, \mathbb{C})$ will be denoted by $\nu= \left(\begin{smallmatrix} \nu_1 & \nu_2 \\ -\nu_2 & \nu_1\end{smallmatrix}\right)$. Given arrows
\[(a_{1},a_{2};b_{1},\dots,b_{n-1};d_{2},\dots,d_{n};j;i_{1},\dots,i_{n-1})\]
as above and $\nu\in {\rm SO}(2,\mathbb{C})$, we set
\begin{equation*}
\begin{pmatrix}a'_{1} \\ a'_{2} \end{pmatrix} = \nu \begin{pmatrix}a_{1} \\ a_{2} \end{pmatrix},
\qquad
\begin{pmatrix}b'_{q} \\ d'_{q+1} \end{pmatrix} = \nu^{-1}\begin{pmatrix}b_{q} \\ d_{q+1} \end{pmatrix}
\qquad\text{for}\quad q=1,\dots,n-1 .
\end{equation*}
The assignment
\begin{gather*}
(a_{1},a_{2};b_{1},\dots,b_{n-1};d_{2},\dots,d_{n};j;i_{1},\dots,i_{n-1}) \\ \qquad{} \longmapsto(a'_{1},a'_{2};b'_{1},\dots,b'_{n-1};d'_{2},\dots,d'_{n};j;i_{1},\dots,i_{n-1}) ,
\end{gather*}
induces an action
\[\widetilde\Phi_n \colon \ {\rm SO}(2,\mathbb{C}) \to \operatorname{Aut}_{\mathbb{C}\hbox{-alg}} (\mathbb{C} Q'_{n}) ,\]
which leaves invariant the generators of the ideal $J'_{n}$, that is,
\begin{equation*}
\widetilde{\Phi}_{n} (\nu) \bigl(J'_{n}\bigr) =J'_{n} .
\end{equation*}
So one has an induced action
\[\Phi_n \colon \ {\rm SO}(2,\mathbb{C}) \to \operatorname{Aut}_{\mathbb{C}\hbox{-alg}} (\Lambda'_{n}) .\]

We wish now to study the space $\operatorname{Rep}(\Lambda'_n, \vec{v}_c, 1) = \mathcal{R}(\Lambda'_n, c)$
of $(c,c)$-dimensional framed representations of $\Lambda'_{n}$ and its open subset
 $\operatorname{Rep}(\Lambda'_n, \vec{v}_c, 1)_{\vartheta_c}^{\rm ss} = \mathcal{R}^{\rm ss}(\Lambda'_n, c)$ of $\vartheta_c$-semistable representations
 (defined analogously to Definition~\ref{def:stabil}).
 For $n=2$ there is nothing new, since
 $\mathcal{R}(\Lambda'_2, c)=\mathcal{R}(\Lambda_2, c)$ and $\mathcal{R}^{\rm ss}(\Lambda'_2, c) = \mathcal{R}^{\rm ss}(\Lambda_2, c)$.
 For $n\geq3$, $\mathcal{R}(\Lambda'_n, c)$ is the affine subvariety of the vector space
\begin{equation*}
\operatorname{Hom}_\mathbb{C}(V_0,V_1)^{\oplus 2}\oplus\operatorname{Hom}_\mathbb{C}(V_1,V_0)^{\oplus 2n-2}\oplus\operatorname{Hom}_\mathbb{C}(V_0,W)\oplus\operatorname{Hom}_\mathbb{C}(W,V_0)^{\oplus n-1}
\end{equation*}
whose points $(A_1,A_2;B_1,\dots,B_{n-1};D_{2},\dots,D_{n},e;f_1,\dots,f_{n-1})$ satisfy the relations
determined by equations~\eqref{eq-gen-J'}, namely,
\begin{equation}
\begin{cases}
A_2D_{q+1}=A_1B_q,\\
D_{q+1}A_2=B_qA_1+f_{q}e
\end{cases}\qquad\text{for}\quad q=1,\dots,n-1 .
\label{eq-(Q1')}\tag{\rm Q$1'$}
\end{equation}

\begin{Lemma}\label{Q-Lemma}
$\mathcal{R}^{\rm ss}(\Lambda'_{n}, c)$ is the open subset of $\mathcal{R}(\Lambda'_{n}, c)$ determined by the conditions:
\begin{enumerate}\itemsep=0pt
 \item[\rm{(Q$2'$)}]
 for all subrepresentations $S=(S_0,S_1)$ such that $S_0\subseteq\ker e$, one has $\dim S_0\leq\dim S_1$, and, if $\dim S_0=\dim S_1$, then $S=0$;
 \item[\rm{(Q$3'$)}]
 for all subrepresentations $S=(S_0,S_1)$ such that $S_0\supseteq\operatorname{Im} f_i$, for $i=1,\dots,n-1$, one has $\dim S_0\leq\dim S_1$.
\end{enumerate}
\end{Lemma}
\begin{proof}
Given a subrepresentation $(S_{0},S_{1})$, we set $s_{i}=\dim S_{i}$, $i=0,1$. By substituting the definitions of $\vec{v}_c$ and $\vartheta_c$ given in Theorem \ref{thm45} into equations~\eqref{eq:stabnak1} and \eqref{eq:stabnak2} one gets
\begin{gather}
 \text{if $S_0\subseteq \ker e$, then $s_{0}\leq s_1-\dfrac{s_1}{2c}$} ;\label{eq:stabnak1'}\\
\text{if $S_0\supseteq \operatorname{Im} f_i$ for $i=1,\dots,n-1$, then $s_{0}\leq s_{1}+\dfrac{1}{2}-\dfrac{s_{1}}{2c}$} .\label{eq:stabnak2'}
\end{gather}
Whenever $s_1>0$, one has $0<\frac{s_1}{2c}<1$; hence, equation~\eqref{eq:stabnak1'} is equivalent to condition \rm{(Q$2'$)}. On the other hand, as $0\leq\frac{1}{2}-\frac{s_{1}}{2c}<\frac{1}{2}$, equation~\eqref{eq:stabnak2'} is equivalent to condition \rm{(Q$3'$)}.
\end{proof}

\begin{Proposition}\label{prop-regular-in-R'}
For each point of $\mathcal{R}^{\rm ss}(\Lambda'_{n}, c)$ the associated matrix pencil $A_{1}+\lambda A_{2}$ is regular.
\end{Proposition}
\begin{proof}
Let $(A_1,A_2;B_1,\dots,B_{n-1};D_{2},\dots,D_{n},e;f_1,\dots,f_{n-1})$ be a point of $\mathcal{R}^{\rm ss}(\Lambda'_{n}, c)$,
and assume that $A_{1}+\lambda A_{2}$ is singular.
If $c=1$, then $A_{1}+\lambda A_{2}$ is singular if and only if $A_{1}=A_{2}=0$. But this implies the subrepresentation
$(V_0,0)$ does not satisfy condition (Q$3'$). Hence we can assume $c\geq 2$.
The fact that the pencil $A_{1}+\lambda A_{2}$ is singular implies that
there is a nontrivial element
\begin{equation}
v(\lambda)=\sum_{\alpha=0}^{\varepsilon}(-\lambda)^{\alpha}v_{\alpha} \in V_0\otimes_{\mathbb{C}}\mathbb{C}[\lambda]
\label{eq-v(l)}
\end{equation}
such that
\begin{equation}
(A_{1}+\lambda A_{2})v(\lambda)=0\qquad\text{for all $\lambda\in\mathbb{C}$.}
\label{eq-Av(l)}
\end{equation}
By arguing as in the proof of \cite[Lemma~4.11]{bblr}, one can show that the minimal degree
polynomial solution $v(\lambda)$ for the pencil $A_{1}+\lambda A_{2}$ has necessarily degree $\varepsilon >0$.
Let us inductively define the vector spaces $\{U_{i}\}_{i\in \mathbb{N}}$ as follows:
\begin{equation*}
 \begin{cases}
 U_{0}=\langle v_{0},\dots,v_{\varepsilon}\rangle, & \\
 U_{2k+1}=A_{1}(U_{2k})+A_{2}(U_{2k}) &\text{for}\quad k\geq0,\\
 \displaystyle U_{2k}=\sum\limits_{q=1}^{n-1}B_{q}(U_{2k-1})+\sum\limits_{q=2}^{n}D_{q}(U_{2k-1}) &\text{for}\quad k\geq1 .
\end{cases}
\end{equation*}
Note that each $U_{j}$, with $j$ even, is a subspace of $V_0$, while each $U_{j}$, with~$j$ odd, is a subspace of~$V_1$. So, if we introduce the subspaces
\begin{equation*}
S_{0}=\sum_{k=0}^{\infty}U_{2k} \subset V_0 ,\qquad S_{1}=\sum_{k=0}^{\infty}U_{2k+1}\subset V_1 ,
\end{equation*}
it follows that $(S_{0},S_{1})$ is a subrepresentation of $(V_0, V_1)$. We will show that this subrepresentation fails to satisfy either condition~(Q$2'$) or condition~(Q$3'$) of Lemma~\ref{Q-Lemma}, so that one gets a~contradiction.

By substituting equation~\eqref{eq-v(l)} into equation~\eqref{eq-Av(l)} one finds out that
\begin{equation}
 \begin{cases}
A_{1}v_{0} =0,\\
A_{1}v_{\alpha} =A_{2}v_{\alpha-1},\qquad \alpha=1,\dots,\varepsilon,\\
A_{2}v_{\varepsilon} =0,
\end{cases}
\label{eq-A_1v=A_2v}
\end{equation}
so that
\begin{equation}
U_{1}=\langle A_{1}v_{1},\dots,A_{1}v_{\varepsilon} \rangle=A_{1}(U_{0}) .
\label{eq-V_1}
\end{equation}

There are two possible cases, either i) $U_{0}\subseteq\ker e$, or ii) $U_{0}\not\subseteq\ker e$.

i) If we suppose that
$U_{0}\subseteq\ker e$,
equation~\eqref{eq-A_1v=A_2v} and condition \eqref{eq-(Q1')} imply that
\begin{equation*}
U_{2}=\sum_{q=2}^{n}D_{q}A_{1}(U_{0}) .
\end{equation*}
By letting $w_{q,\alpha}=D_{q}A_{1}v_{\alpha}$, $\alpha=1,\dots,\varepsilon$,
for each $q=2,\dots,n$ we obtain an
 element \[(w_{q,1},\dots,w_{q,\varepsilon})\in U_2^{\oplus\varepsilon} \]
 such that $\sum\limits_{\alpha=0}^{\varepsilon-1}(-\lambda)^{\alpha}w_{q,\alpha+1}$ is a polynomial solution for the pencil $A_{1}+\lambda A_{2}$ of degree $\varepsilon - 1$. Since we have supposed $\varepsilon$ to be minimal, one has
 $(w_{q,1},\dots,w_{q,\varepsilon}) = 0$. From that it is easy to deduce that $U_2 =0$ and that $(S_0, S_1) = (U_0,U_1)$.
 So, since $\ker A_1 \cap U_0 \neq 0$ by equation~\eqref{eq-A_1v=A_2v}, then
 equation~\eqref{eq-V_1} entails that $(S_{0},S_{1})$ is a subrepresentation violating condition (Q$2'$).

 ii) Suppose now that $U_{0}$ is not contained in $\ker e$. So, there is at least one $\gamma\in \{0,\dots, \varepsilon\}$ such that
 $e(v_\gamma) \neq 0$. Condition \eqref{eq-(Q1')} implies that
 \begin{equation}
\operatorname{Im} f_{q}= \langle f_{q}e(v_{\gamma})\rangle \subseteq U_{2}\qquad\text{for all}\quad q=1,\dots,n-1 .
\label{eq-imf-in-V}
\end{equation}
To simplify computations, we may assume $\gamma =0$ and $e(v_0)=1$.
 Actually, one checks that the ${\rm SO}(2,\mathbb{C})$ action on $\Lambda'_n$ induces an action on $\mathcal{R}(\Lambda'_{n}, c)$, which commutes with the $G_{\vec{v}_c}$ action defined on the same space, and therefore it restricts to an ${\rm SO}(2,\mathbb{C})$ action on $\mathcal{R}^{\rm ss}(\Lambda'_{n}, c)$. Moreover, this action preserves the regularity of the matrix pencil
 $A_1 + \lambda A_2$. An element $\nu= \left(\begin{smallmatrix} \nu_1 & \nu_2 \\ -\nu_2 & \nu_1\end{smallmatrix}\right)\in {\rm SO}(2,\mathbb{C})$ produces a change of basis
 \[(v_0, \dots, v_{\varepsilon}) \mapsto \nu \cdot (v_0, \dots, v_{\varepsilon}) = (v'_0, \dots, v'_{\varepsilon}) ,\]
 so that
\begin{equation*}
e(v'_{0})=\sum_{\alpha=0}^{\varepsilon}(-\nu_{2})^{\alpha}\nu_{1}^{\varepsilon-\alpha}e(v_{\alpha}) .
\end{equation*}
Since $(e(v_{0}),\dots,e(v_{\varepsilon}))\neq(0,\dots,0)$, there is $\nu\in {\rm SO}(2, \mathbb{C})$ so that $e(v'_{0})\neq0$. Moreover, $e(v'_0)$ can be assumed to be $1$.

Next, by using condition \eqref{eq-(Q1')} and equation~\eqref{eq-A_1v=A_2v}, along with the identity
$A_{2}(\operatorname{Im} f_{q}) =\langle A_{2}f_{q}(1)\rangle$ $ = \langle A_{2}f_{q}(e(v_0))\rangle$, it is not hard to show that
\begin{equation}
A_{2}(\operatorname{Im} f_{q})\subseteq A_{1}(U_{2})\qquad\text{for all}\quad q=1,\dots,n-1 .
\label{eq-A_2(im-f)}
\end{equation}

Now we show that
\begin{equation}
U_{2k+1}\subseteq \sum_{l=1}^{k}A_{1}(U_{2l})
\label{eq-V_(2k+1)-in-imA_{1}}
\end{equation}
for all $k\geq 1$.
Assume $k=1$. By using equations~\eqref{eq-A_1v=A_2v}, \eqref{eq-(Q1')} and condition $e(v_{0})=1$, one gets
\begin{equation*}
f_{q}e(v_{\alpha})=e(v_{\alpha})D_{q+1}A_{1}v_{1}
\end{equation*}
for $q=1,\dots,n-1$. Hence, by using equations~\eqref{eq-A_1v=A_2v} and \eqref{eq-(Q1')} again one shows that
\begin{equation*}
U_{2}= \sum_{q=2}^{n}D_{q}A_{1}(U_{0}) .
\end{equation*}
 Then $U_{3}$ is spanned by the sets of vectors
\begin{equation*}
\{A_{1}D_{q}A_{1}v_{\alpha}\}_{\begin{subarray}{l} q=2,\dots,n \\ \alpha=1,\dots,\varepsilon \end{subarray}}\subseteq A_{1}(U_{2}) ,\qquad
\{A_{2}D_{q}A_{1}v_{\alpha}\}_{\begin{subarray}{l} q=2,\dots,n \\ \alpha=1,\dots,\varepsilon \end{subarray}}\subseteq A_{2}(U_{2})
\end{equation*}
and it follows directly from equations~\eqref{eq-(Q1')} that
$A_{2}D_{q}A_{1}v_{\alpha} \in A_{1}(U_{2})$, for $q=2,\dots,n$. So $U_3\subseteq A_1(U_2)$.

Let us now suppose that equation~\eqref{eq-V_(2k+1)-in-imA_{1}} holds true for $1\leq k \leq m$, with $m\geq 1$. This means that $U_{2m+1}$ is spanned by vectors of the form
$A_{1}w$ with $w\in U_{2l}$, $l=1,\dots,m$.
By noticing that $U_{2m+2}$ is spanned by vectors of the form $B_{p}A_{1}w$ and $D_{q}A_{1}w'$, with $w\in U_{2l}$ and $w'\in U_{2l'}$ for $l,l'=1,\dots,m$, and by using equation~\eqref{eq-(Q1')} and the inductive hypothesis one
finds out that
\begin{equation*}
U_{2m+2}\subseteq \sum_{l=1}^{m} \sum_{q=2}^{n}D_{q}A_{1}(U_{2l})+\sum_{q=1}^{n-1}\operatorname{Im} f_{q} .
\end{equation*}
From this it follows
\begin{align}
A_1(U_{2m+2}) + A_2(U_{2m+2}) &= U_{2m+3} \nonumber\\
& \subseteq A_{1}(U_{2(m+1)})+\sum_{l=1}^{m} \sum_{q=2}^{n}A_{2}D_{q}A_{1}(U_{2l}) +\sum_{q=1}^{n-1}A_{2}(\operatorname{Im} f_{q}) \nonumber\\
& \subseteq A_{1}(U_{2(m+1)})+A_{1}(U_{2})+\sum_{l=1}^{m} \sum_{q=2}^{n}A_{2}D_{q}A_{1}(U_{2l}) ,\label{last-step}
\end{align}
where in the last step equation~\eqref{eq-A_2(im-f)} has been used.
For $q=2,\dots, n$, equation~\eqref{eq-(Q1')} implies that
\begin{equation*}
A_{2}D_{q}A_{1}(U_{2l}) \subseteq A_{1}(U_{2(l+1)}) .
\end{equation*}
Thus, from equation~\eqref{last-step} we may conclude that
\begin{equation*}
U_{2(m+1)+1}\subseteq\sum_{l=1}^{m+1}A_{1}(U_{2l}) ,
\end{equation*}
so that the inclusion \eqref{eq-V_(2k+1)-in-imA_{1}} is proved.

This and equation~\eqref{eq-V_1} imply that
\begin{equation*}
S_{1}=\sum_{k=0}^{\infty}U_{2k+1}\subseteq A_{1}(U_{0})+\sum_{k=1}^{\infty}\sum_{l=1}^{k}A_{1}(U_{2l})=\sum_{k=0}^{\infty}A_{1}(U_{2k})= A_{1}\left(\sum_{k=0}^{\infty}U_{2k}\right)=A_{1}(S_{0}).
\end{equation*}
But $A_{1}(S_{0})\subseteq S_{1}$, so that
\begin{equation*}
S_{1}=A_{1}(S_{0}) .
\end{equation*}
By equation~\eqref{eq-A_1v=A_2v}, one has $\ker A_{1}\cap S_{0}\neq0$, and therefore $\dim S_{1}<\dim S_{0}$.
Finally, equation~\eqref{eq-imf-in-V} implies that the subrepresentation $(S_{0},S_{1})$ violates condition (Q$3'$).
 \end{proof}

When $n\geq 3$, there is a map $\mathcal{R}(\Lambda_{n}, c) \longrightarrow \mathcal{R}(\Lambda'_{n}, c)$ given by
\begin{equation*}
(A_1,A_2;C_1,\dots,C_n;e;f_1,\dots,f_{n-1}) \mapsto
(A_1,A_2;C_1,\dots,C_{n-1};C_{2},\dots,C_{n},e;f_1,\dots,f_{n-1}) .
\end{equation*}
This map provides a $G_{\vec{v}_{c}}$-equivariant isomorphism of $\mathcal{R}(\Lambda_{n}, c)$ onto the subvariety of $\mathcal{R}(\Lambda'_{n}, c)$ cut by the equations
\begin{equation*}
B_{q}=D_{q}\qquad\text{for}\quad q=2,\dots,n-1
\end{equation*}
(cf.\ equations~\eqref{eq-gen-K_n}). Through this isomorphism $\mathcal{R}(\Lambda_{n}, c)$ may be regarded as a closed subvariety of $\mathcal{R}(\Lambda'_{n}, c)$.
\begin{Lemma}\label{lm-R_n(c)-in-R'_n(c)}
When $n\geq 3$, one has that
\begin{equation*}
\mathcal{R}^{\rm ss}(\Lambda_{n}, c) =\mathcal{R}^{\rm ss}(\Lambda'_{n}, c) \cap \mathcal{R}(\Lambda_{n}, c) .
\end{equation*}
\end{Lemma}
\begin{proof}Semistability is a numerical condition which is to be checked on the set of all submodules of a given representation. Hence, it is enough to show that for any left $\Lambda_{n}$-module $M$, an abelian subgroup $N\subset M$ is a left $\Lambda_{n}$-submodule if and only if it is a left $\Lambda'_{n}$-submodule (notice that $M$ has also a natural structure of left $\Lambda'_{n}$-module, induced by restriction of scalars; cf.\ Lemma~\ref{lm-Lambda'-->>Lambda}).
However, precisely because the algebra $\Lambda_n$ is a quotient of $\Lambda'_n$,
the category $\Lambda_{n}$-{\bf mod} is a full subcategory of $\Lambda'_{n}$-{\bf mod}, and this implies in particular that the set of all subobjects of a given $\Lambda_n$-module is the same in the two categories.
\end{proof}

\begin{Theorem}\label{mainthm} The component $\mathcal{H} (n,c) $ of the moduli
space $\mathcal{M}(\Lambda_n,\vec{v}_c, 1, \vartheta_c)$ defined by equations~\eqref{sliceeqs}
coincides with the whole of $\mathcal{M}(\Lambda_n,\vec{v}_c, 1, \vartheta_c)$.
\end{Theorem}
\begin{proof}For each representation $(A_1,A_2;C_1, \dots,C_n;e;f_1,\dots,f_{n-1}) \!\in\! \mathcal{R}^{\rm ss}(\Lambda_n,c)$, equations~\eqref{sliceeqs} hold if and only if the pencil $A_1 + \lambda A_2$ is regular (condition (P2) in Section~\ref{background} and in \cite{bblr}):
\begin{itemize}\itemsep=0pt
 \item in the proof of Proposition~4.9 of~\cite{bblr} it has been shown that condition (P2) holds in $Z_n(c)=\operatorname{pr}^{-1}(\mathcal{H} (n,c))$; i.e., equations~\eqref{sliceeqs} imply the regularity of the pencil;
 \item further on, in the proof of Theorem 4.5 of \cite{bblr} it has been shown that $Z_n(c)$ actually coincides with the open subset of $\mathcal{R}^{\rm ss}(\Lambda_n,c)$ (denoted $\mathcal{R}_n(c)$ in \cite{bblr}) where condition (P2) is satisfied; i.e., the regularity of the pencil implies equations~\eqref{sliceeqs}.
 \end{itemize}
So, the orbit of a $\vartheta_c$-semistable representation \[(A_1,A_2;C_1, \dots,C_n;e;f_1,\dots,f_{n-1}) \in \mathcal{R}^{\rm ss}(\Lambda_n,c)\] lies in $\mathcal{H}(n,c)$ if and only if the pencil $A_1 + \lambda A_2$ is regular. Then the conclusion follows from Proposition \ref{prop-regular-in-R'} and Lemma \ref{lm-R_n(c)-in-R'_n(c)}.
\end{proof}

\section{A remark involving the 2-Kronecker quiver}
We want to rephrase Proposition \ref{prop-regular-in-R'} is a slightly different way which involves the Kronecker quiver
 with two arrows $Q_{K}$
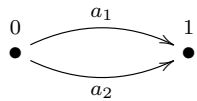
$$
 \xymatrix@R-2.3em{
&\mbox{\scriptsize$0$}&&\mbox{\scriptsize$1$}\\
&\bullet\ar@/^2ex/[rr]^{a_{1}}\ar@/_2ex/[rr]_{a_{2}}&&\bullet
}
$$
The new claim, Proposition \ref{relGIT}, may be regarded as a statement in relative Geometric Invariant Theory.

The vector space of $\vec{v}_c=(c,c)$-dimensional representations of $Q_{K}$ is the space
$\operatorname{Rep}(Q_K, \vec{v}_c) = \operatorname{Hom}_{\mathbb{C}}( V_0, V_1)^{\oplus 2}$.
Since Definition \ref{def:stabil} only applies to framed quivers, we need a slightly different notion of semistability.
So we recall from \cite{KiQ, Rud97} that, given $\vartheta\in\mathbb R^2$, a $\vec{v}_c$-dimensional representation of $Q_K$ is said to be \emph{$\vartheta$-semistable} if, for any proper nontrivial subrepresentation supported by $(S_0,S_1)$ {$\subseteq (V_0,V_1)$}, one has
\begin{equation}\label{stabkro}
\frac{\vartheta\cdot(\dim S_0,\dim S_1)}{\dim S_0 +\dim S_1} \leq
\frac{\vartheta\cdot \vec{v}_c}{2 c } .
\end{equation}
A $\vartheta$-semistable representation is \emph{$\vartheta$-stable} if strict inequality holds in \eqref{stabkro}.

As in Section \ref{sectionmainresult}, we set $\vartheta_c = (2c, 1-2c)$.
\begin{Lemma}
\label{lm-Q_K-ss}
A point $(A_{1},A_{2})\in \operatorname{Rep}(Q_K, \vec{v}_c)$ is $\vartheta_{c}$-semistable if and only if
the matrix pencil $A_1+ \lambda A_2$ is regular.
\end{Lemma}
\begin{proof}
Let $(A_1, A_2)$ be a representation of $Q_K$ supported by the pair of vector spaces $(V_0,V_1)$, and consider a proper subrepresentation supported by $(S_0,S_1)$.
If the stability parameter is $\vartheta_c = (2c, 1-2c)$, the inequality \eqref{stabkro} is equivalent to
\begin{equation*}
\frac{2c \dim S_0 + (1-2c)\dim S_1}{\dim S_0 +\dim S_1} \leq \frac{1}{2} ,
\end{equation*}
which is in turn equivalent to
\begin{equation}
\dim S_{0}\leq \dim S_{1} .
\label{eq-mu_k(S)'}
\end{equation}
It is not hard to show that \eqref{eq-mu_k(S)'} implies
\begin{equation}\label{eq-dimA(S)}
\dim(A_{1}(S)+A_{2}(S))\geq\dim S \quad\text{for all vector subspaces $S\subseteq V_{0}$.}
\end{equation}
 Conversely, if condition \eqref{eq-dimA(S)} is satisfied, then, given any subrepresentation supported by $S=(S_0, S_1)$,
 one has
 \begin{equation*}
\dim S_{1}\geq \dim(A_{1}(S_{0})+A_{2}(S_{0})) \geq \dim S_{0} .
\end{equation*}
Finally, by \cite[Lemma 4.10]{bblr} condition \eqref{eq-dimA(S)} is equivalent to the fact that
the matrix pencil $A_1+ \lambda A_2$ is regular.
\end{proof}

Recall that $\mathcal{R}(\Lambda_n, c)$ is the affine subvariety of
\[
\operatorname{Rep}(Q_n, \vec{v}_c, 1) = \operatorname{Hom}_\mathbb{C}(V_0,V_1)^{\oplus2}\oplus\operatorname{Hom}_\mathbb{C}(V_1,V_0)^{\oplus n}\oplus\operatorname{Hom}_\mathbb{C}(V_0,W)\oplus\operatorname{Hom}_\mathbb{C}(W,V_0)^{\oplus n-1}
\]
defined by equations \eqref{eq-gen-Q1}.
Let us denote by $\pi_n \colon \mathcal{R}(\Lambda_n, c) \to \operatorname{Rep}(Q_K, \vec{v}_c)$ the restriction of the natural projection
$\operatorname{Rep}(Q_n, \vec{v}_c, 1) \to \operatorname{Hom}_\mathbb{C}(V_0,V_1)^{\oplus2}= \operatorname{Rep}(Q_K, \vec{v}_c)$.

As a straightforward consequence of Lemma \ref{lm-Q_K-ss}, Proposition \ref{prop-regular-in-R'} may be rephrased in the following terms.
\begin{Proposition}\label{relGIT} Each $(\vec{v}_c, 1)$-dimensional $\vartheta_c$-semistable representation of $\Lambda_n$ is mapped by $\pi_n$ to a $\vec{v}_c$-dimensional $\vartheta_c$-semistable representation of $Q_K$:
\[\pi_n\bigl( \mathcal{R}^{\rm ss}(\Lambda_n, c)\bigr) \subseteq \operatorname{Rep}(Q_K, \vec{v}_c)_{\vartheta_c}^{\rm ss} .\]
\end{Proposition}

\subsection*{Acknowledgements}

This research was partly supported by INdAM-GNSAGA and PRIN ``Geometria delle variet\`a algebriche''. V.L.~was supported by the FAPESP Postdoctoral Grants No. 2015/07766-4 (Bolsa no Pa\'is) and 2017/22052-9 (Bolsa Est\'agio de Pesquisa no Exterior). This research was partly carried out while he was visiting the Department of Mathematics of the Northeastern University (Boston), and he wishes to thank that institution for hospitality. We thank the referees for their valuable suggestions which allowed us to substantially improve the presentation.

\pdfbookmark[1]{References}{ref}
\LastPageEnding


\begin{thebibliography}{99}
\footnotesize\itemsep=0pt

\bibitem{bblr}
Bartocci C., Bruzzo U., Lanza V., Rava C.L.S., Hilbert schemes of points of
 {${\mathcal O}_{{\mathbb P}^1}(-n)$} as quiver varieties, \href{https://doi.org/10.1016/j.jpaa.2016.12.012}{\textit{J.~Pure
 Appl. Algebra}} \textbf{221} (2017), 2132--2155, \href{https://arxiv.org/abs/1504:02987}{arXiv:1504:02987}.

\bibitem{bbr}
Bartocci C., Bruzzo U., Rava C.L.S., Monads for framed sheaves on {H}irzebruch
 surfaces, \href{https://doi.org/10.1515/advgeom-2014-0027}{\textit{Adv. Geom.}} \textbf{15} (2015), 55--76, \href{https://arxiv.org/abs/1205.3613}{arXiv:1205.3613}.

\bibitem{blr}
Bartocci C., Lanza V., Rava C.L.S., Moduli spaces of framed sheaves and quiver
 varieties, \href{https://doi.org/10.1016/j.geomphys.2016.10.011}{\textit{J.~Geom. Phys.}} \textbf{118} (2017), 20--39,
 \href{https://arxiv.org/abs/1610:02731}{arXiv:1610:02731}.

\bibitem{BFMT}
Bruzzo U., Fucito F., Morales J.F., Tanzini A., Multi-instanton calculus and
 equivariant cohomology, \href{https://doi.org/10.1088/1126-6708/2003/05/054}{\textit{J.~High Energy Phys.}} \textbf{2003} (2003),
 no.~5, 054, 24~pages, \href{https://arxiv.org/abs/hep-th/0211108}{arXiv:hep-th/0211108}.

\bibitem{CrBo}
Crawley-Boevey W., Geometry of the moment map for representations of quivers,
 \href{https://doi.org/10.1023/A:1017558904030}{\textit{Compositio Math.}} \textbf{126} (2001), 257--293.

\bibitem{Dorey}
Dorey N., Hollowood T.J., Khoze V.V., Mattis M.P., The calculus of many
 instantons, \href{https://doi.org/10.1016/S0370-1573(02)00301-0}{\textit{Phys. Rep.}} \textbf{371} (2002), 231--459,
 \href{https://arxiv.org/abs/hep-th/0206063}{arXiv:hep-th/0206063}.

\bibitem{Fogarty}
Fogarty J., Algebraic families on an algebraic surface, \href{https://doi.org/10.2307/2373541}{\textit{Amer.~J. Math.}}
 \textbf{90} (1968), 511--521.

\bibitem{Ginz}
Ginzburg V., Lectures on {N}akajima's quiver varieties, in Geometric Methods in
 Representation Theory.~{I}, \textit{S\'{e}min. Congr.}, Vol.~24, Soc. Math.
 France, Paris, 2012, 145--219, \href{https://arxiv.org/abs/0905:0686}{arXiv:0905:0686}.

\bibitem{KiQ}
King A.D., Moduli of representations of finite-dimensional algebras,
 \href{https://doi.org/10.1093/qmath/45.4.515}{\textit{Quart.~J. Math. Oxford}} \textbf{45} (1994), 515--530.

\bibitem{Kuz}
Kuznetsov A., Quiver varieties and {H}ilbert schemes, \href{https://doi.org/10.17323/1609-4514-2007-7-4-673-697}{\textit{Mosc. Math.~J.}}
 \textbf{7} (2007), 673--697, \href{https://arxiv.org/abs/math.AG/0111092}{arXiv:math.AG/0111092}.

\bibitem{Naka-ALE}
Nakajima H., Instantons on {ALE} spaces, quiver varieties, and {K}ac--{M}oody
 algebras, \href{https://doi.org/10.1215/S0012-7094-94-07613-8}{\textit{Duke Math.~J.}} \textbf{76} (1994), 365--416.

\bibitem{Nakabook}
Nakajima H., Lectures on {H}ilbert schemes of points on surfaces,
 \textit{University Lecture Series}, Vol.~18, \href{https://doi.org/10.1090/ulect/018}{Amer. Math. Soc.}, Providence,
 RI, 1999.

\bibitem{Naka-ring}
Nakajima H., Introduction to quiver varieties~-- for ring and representation
 theorists, in Proceedings of the 49th {S}ymposium on {R}ing {T}heory and
 {R}epresentation {T}heory, Symp. Ring Theory Represent. Theory Organ. Comm.,
 Shimane, 2017, 96--114, \href{https://arxiv.org/abs/1611:10000}{arXiv:1611:10000}.

\bibitem{Niki}
Nekrasov N.A., Seiberg--{W}itten prepotential from instanton counting,
 \href{http://projecteuclid.org/euclid.atmp/1111510432}{\textit{Adv. Theor. Math. Phys.}} \textbf{7} (2003), 831--864,
 \href{https://arxiv.org/abs/hep-th/0206161}{arXiv:hep-th/0206161}.

\bibitem{Ngo-Siv}
Ngo N.V., \v{S}ivic K., On varieties of commuting nilpotent matrices,
 \href{https://doi.org/10.1016/j.laa.2014.03.032}{\textit{Linear Algebra Appl.}} \textbf{452} (2014), 237--262,
 \href{https://arxiv.org/abs/1308:4438}{arXiv:1308:4438}.

\bibitem{Ok}
Okonek C., Schneider M., Spindler H., Vector bundles on complex projective
 spaces, \textit{Progress in Mathematics}, Vol.~3, \href{https://doi.org/10.1007/978-3-0348-0151-5}{Birkh\"{a}user}, Boston,
 Mass., 1980.

\bibitem{Rud97}
Rudakov A., Stability for an abelian category, \href{https://doi.org/10.1006/jabr.1997.7093}{\textit{J.~Algebra}} \textbf{197}
 (1997), 231--245.

\end{thebibliography}
\end{document}